\documentclass[11pt]{article}
\usepackage{amsmath,amssymb,amsbsy,amsfonts,amsthm,latexsym,
            amsopn,amstext,amsxtra,euscript,amscd,amsthm,lipsum,hyperref}
\usepackage[]{caption,multirow}
\usepackage[lined,boxed,commentsnumbered,ruled,linesnumbered]{algorithm2e}
\usepackage{longtable,geometry,setspace}

\newtheorem{rem}{Remark}
\newtheorem{remark}[rem]{Remark}

\newtheorem{thm}{Theorem}
\newtheorem{theorem}[thm]{Theorem}

\newtheorem{conj}{Conjecture}

\newtheorem{exam}{Example}
\newtheorem{example}[exam]{Example}

\def\\{\cr}
\def\({\left(}
\def\){\right)}
\def\[{\left[}
\def\]{\right]}
\def\<{\langle}
\def\>{\rangle}

\def\cA{{\mathcal A}}
\def\cB{{\mathcal B}}

\def\F{{\mathbb F}}

\begin{document}

\title{\bf Representing the inverse map as a composition of quadratics in a finite field of characteristic $2$}

\author{\Large  Florian Luca$^{1,2}$,    \Large  Santanu Sarkar$^3$,   Pantelimon St\u anic\u a$^4$\\
\vspace{0.3cm}\\  
\small $^1$ School of Mathematics, University of the Witwatersrand,\\
\small  Private Bag X3, Wits 2050, Johannesburg, South Africa; and\\   
\small $^2$ Centro de Ciencias Matem\'aticas, UNAM,\\
\small Morelia, Mexico; {\tt Florian.Luca@wits.ac.za},\\
\small $^3$ Department of Mathematics, Indian Institute of Technology Madras,\\
\small Sardar Patel Road, Chennai TN 600036, INDIA; {\tt santanu@iitm.ac.in},\\
\small $^4$ Applied Mathematics Department, Naval Postgraduate School,\\
\small Monterey 93943, USA; {\tt pstanica@nps.edu}
}

\date{\today}

\pagenumbering{arabic}

\maketitle

\begin{abstract}
In 1953, Carlitz~\cite{Car53}  showed that  all permutation polynomials over $\F_q$, where $q>2$ is a power of a prime,  are generated by the special permutation polynomials $x^{q-2}$ (the inversion) and $ ax+b$ (affine functions, where $0\neq a, b\in  \F_q$). Recently, Nikova, Nikov and Rijmen~\cite{NNR19} proposed an algorithm (NNR) to find a decomposition of the inverse function in quadratics, and computationally covered all dimensions $n\leq 16$. Petrides~\cite{P23} found a class of integers for which it is easy to decompose the inverse into quadratics, and improved the NNR algorithm, thereby extending the computation up to $n\leq 32$. Here,  we extend Petrides' result, as well as we propose a number theoretical approach, which allows us to cover easily all (surely, odd) exponents up to~$250$, at least.
\end{abstract}

\section{Introduction}

\let\thefootnote\relax\footnotetext{This is an expanded and vastly improved version of the Extended Abstract, which was presented at the 8th International Workshop on Boolean Functions and their Applications (BFA) in Voss in September 2023.} 

In~\cite{Car53},
Carlitz showed that  all permutation polynomials over $\F_q$, where $q>2$ is a power of a prime,  are generated by the special permutation polynomials $x^{q-2}$ (the inversion) and $ ax+b$
(affine functions, where $0\neq a, b\in  \F_q$). The smallest number of inversions in such a decomposition is called the {\em Carlitz rank}.

Here, we ask whether the inverse in ${\mathbb F}_{2^n}$ (the finite field of dimension $n$ over the two-element prime field ${\mathbb F}_2$) can be written as a composition of quadratics, or quadratics and cubics. That is, we ask if there are integers $r\ge 1$ and $a_1\ge 0,\ldots,a_r\ge 0$ such that
$$
-1\equiv \prod_{i=1}^r (2^{a_i}+1)\pmod {2^n-1}.
$$  
Nikova, Nikov and Rijmen~\cite{NNR19} proposed an algorithm to find such a decomposition. Via Carlitz~\cite{Car53}, they were able to use the algorithm and show that for $n\leq 16$ any permutation can be decomposed in quadratic permutations, when $n$ is not multiple of $4$ and in cubic permutations, when $n$ is multiple of $4$. Petrides~\cite{P23}, in addition to a theoretical result, which we will discuss below, improved the complexity of the algorithm and presented a computational table of shortest decompositions for $n\leq 32$, allowing also cubic permutations in addition to quadratics. Here, we find a number theoretical approach which allows us to cover all (surely, odd) exponents up to~$250$.

\section{Our results}
Let $\nu_2$ be the 2-valuation, that is, the largest power of $2$ dividing the argument.
We start with a proposition,   extending one of Petrides' results~\cite{P23}, which states that if $n$ is an odd integer and $\frac{n-1}{2^{\nu_2(n-1)}}\equiv 2^k\pmod {2^n-1}$, for some $k$, then
\begin{align*}
2^n-2&=2\left(2^{\left(\frac{n-1}{2^{\nu_2(n-1)}}\right) 2^{\nu_2(n-1)}}-1\right)\\
&=2\left(2^{\frac{n-1}{2^{\nu_2(n-1)}}}-1\right)\prod_{j=1}^{\nu_2(n-1)}\left(2^{\frac{n-1}{2^j}}+1\right)\\
&\equiv 2\left(2^{2^k}-1\right)\prod_{j=1}^{\nu_2(n-1)}\left(2^{\frac{n-1}{2^j}}+1\right)\\
& = 2\prod_{j=0}^{k-1}\left(2^{2^j}+1\right)\prod_{j=1}^{\nu_2(n-1)}\left(2^{\frac{n-1}{2^j}}+1\right).
\end{align*}
This implies, via Carlitz~\cite{Car53}, that for all odd integers (coined {\em good integers}, with the counterparts named {\em bad integers} in~\cite{M97}) satisfying the  congruence $\frac{n-1}{2^{\nu_2(n-1)}}\equiv 2^k\pmod {2^n-1}$, one can decompose any permutation polynomial in $\F_{2^n}$  into affine and quadratic
power permutations. 

The smallest odd positive integer that is not {\em good} is $n=7$. We note however that in that case 
$$2^7-2=2(2^6-1)=2(2^2-1)(2^4+2^2+1)=2(2+1)(2^4+2^2+1),$$
and so, any permutation in ${\mathbb F}_{2^7}$ can be decomposed into affine, quadratic and cubic permutations. We are ready to generalize this observation.
\begin{theorem}
Let $n$ be an odd  integer satisfying 
$$
\frac{n-1}{2^{\nu_2(n-1)}}\equiv 2^k3^s\pmod {2^n-1},
$$
for some non-negative integers $r,s$. Then, the inverse power permutation in $\F_{2^n}$ has a decomposition into affine,  quadratic  and cubic power permutations of length~$k+s+\nu_2(n-1)$.
\end{theorem}
\begin{proof}
As we have already alluded to above, using the difference of cubes factorization, $a^3-b^3=(a-b)(a^2+ab+b^2)$, we have
\allowdisplaybreaks
\begin{align*}
2^n-2&=2\left(2^{\frac{n-1}{2^{\nu_2(n-1)}}}-1\right)\prod_{j=1}^{\nu_2(n-1)}\left(2^{\frac{n-1}{2^j}}+1\right)\\
&\equiv 2 \left(2^{2^k3^s}-1\right)
\prod_{j=1}^{\nu_2(n-1)}\left(2^{\frac{n-1}{2^j}}+1\right)
\\
&= 2\left(2^{2^k 3^{s-1}}-1\right) \left(2^{2^{k+1}3^{s-1}}+2^{2^k3^{s-1}}+1\right) \prod_{j=1}^{\nu_2(n-1)}\left(2^{\frac{n-1}{2^j}}+1\right)
\\
&\cdots\cdots\cdots\cdots\\
&=2 \left(2^{2^k}-1\right)\prod_{j=0}^{s-1} \left(2^{2^{k+1}3^j}+2^{2^k 3^j}+1\right)
\prod_{j=1}^{\nu_2(n-1)}\left(2^{\frac{n-1}{2^j}}+1\right)
\\
& \equiv 2\prod_{j=0}^{k-1}\left(2^{2^j}+1\right) \prod_{j=0}^{s-1} \left(2^{2^{k+1}3^j}+2^{2^k 3^j}+1\right)\prod_{j=1}^{\nu_2(n-1)}\left(2^{\frac{n-1}{2^j}}+1\right).
\end{align*}
The claim is shown.
\end{proof}
\begin{example}
It is natural to investigate the counting function~$\cB(x)$ of {\em superbad integers} (that is,  integers~$n$ such that 
$\frac{n-1}{2^{\nu_2(n-1)}}\not\equiv 2^k3^s\pmod {2^n-1}$), with $\cB(x)=\{n\leq x\,:\, n\text{ is superbad}\}$, or its complement 
$$
\cA(x)=\{n\leq x\,:\, \frac{n-1}{2^{\nu_2(n-1)}}\equiv 2^k3^s\pmod {2^n-1}\}.
$$
\begin{sloppypar}
As an example, $|\cB(50)|=16$, more precisely, 
$$\cB(50)=\{1, 2, 3, 4, 5, 7, 9, 10, 13, 17, 19, 25, 28, 33, 37, 49\}$$
(Petrides~\textup{\cite{P23}} noted that $25$ integers up to $50$ are bad, so our extension surely prunes the integers  better). In a recent paper~\textup{\cite{LS23}}, it was shown that 
$$\cA(x)\ll \frac{x}{(\log\log x)^{1+o(1)}}.
$$
\end{sloppypar}
\end{example}

For most of the remaining work we restrict our attention to $n=p$, a prime. 

Let $p\ge 3$ be prime, $N:=N_p=2^p-1$. It is known that if $q\mid N_p$, then $q\equiv 1\pmod {p}$. 
We ask if we can say anything about the number of distinct prime factors $\omega(N_p)$ of $N_p$. We propose the following conjecture.

\begin{conj}
\label{conj:1}
There exists $p_0$ such that for $p>p_0$, $\omega(N_p)<1.36\log p$.
\end{conj} 

Similar heuristics regarding lower bounds for $\Omega(2^n-1)$ and $\omega(2^n-1)$ can be found in \cite{KoLa} and \cite{LuSt}. 

Our Conjecture \ref{conj:1}, is based on statistical arguments originating from sieve methods. The Tur\'an-Kubilius inequality asserts that 
$$
\sum_{n\le x} (\omega(n)-\log\log x)^2=O(x\log\log x).
$$
So, if $\delta>0$ is fixed, the set of $n\le x$ such that $\omega(n)\ge (1+\delta)\log\log x$ is of counting function $O_{\delta}(x/\log\log x)$. One can do better using sieves. Indeed, Exercise  04 in \cite{HT} shows that for fixed $\delta>0$ we in fact have that 
$$
\#\{n\le x: \omega(n)\ge (1+\delta)\log\log x\}\ll_{\delta} \frac{x}{(\log x)^{Q(\delta)}},
$$
where $Q(\delta):=(1+\delta)\log((1+\delta)/e)+1$. We would like to apply such heuristics to $N_p=2^p-1$. But note that if $q\mid N_p$, then $2^p\equiv 1\pmod q$. In particular, 
${\displaystyle{\left(\frac{2}{q}\right)=1}}$, so $q\equiv \pm 1\pmod 8$. But then the same proof as Exercise 04 in \cite{HT} shows that 
\begin{eqnarray*}
& \# & \{n\le x: q\mid n \Rightarrow q\equiv \pm 1\pmod 8   \quad {\text{\rm and}} \quad \omega(n)\ge (1+\delta)\log\log x\}\\
& \le & \frac{x}{(\log x)^{Q_1(\delta)+o(1)}}
\end{eqnarray*}
as $x\to\infty$, where $Q_1(\delta):=(1+\delta)\log((1+\delta)/(0.5 e))+1$. Taking $\delta=0.36$, we get $Q_1(\delta)=1.00086\ldots$. Thus, the probability that a number having only 
prime factors congruent to $\pm 1\pmod 8$ to have more than $1.36\log\log n$ distinct prime factors is
$$
O\left(\frac{1}{(\log n)^{1.00008}}\right). 
$$
Applying this to $N_p$, we get 
$$
O\left(\frac{1}{(\log(2^p-1))^{1.0008}}\right)\ll \frac{1}{p^{1.0008}},
$$
and since the series   
$$
\sum_{p\ge 3} \frac{1}{p^{1.0008}}
$$  
is convergent, we are led to believe that maybe there are at most finitely many prime numbers $p$ such that $\omega(N_p)\ge 1.36\log p$. It has been noted that perhaps infinitely often $\omega(N_p)\ge 2$.  For example, this is the case if $p\equiv 3\pmod 4$ is such that $q=2p+1$ is prime. Indeed, then $2$ is a quadratic residue modulo $q$ so $2^{(q-1)/2}\equiv 1\pmod q$, showing that $q\mid N_p$. Since $N_p$ is never a perfect power, in particular it cannot be a power of $q$, we get the desired conclusion that $\omega(N_p)\ge 2$.  

\begin{conj}
\label{conj:2}
There exists $p_0$ such that if $p>p_0$, then $N _p$ is squarefree.
\end{conj}

We offer some heuristic evidence for Conjecture \ref{conj:2}. 
Knowing that the prime $q$ divides $N_p$, the conditional probability that $N_p$ is divisible by $q^2$ is $1/q$. Thus, the probability that $N_p$ is not squarefree is bounded above by 
$$
\sum_{q: q\mid N_p~{\text{\rm for~some~prime}}~p} \frac{1}{q},
$$
and it was shown by Murata and Pomerance in \cite{MuPo} that the above sum is finite under GRH. 

So, assuming Conjecture \ref{conj:1} and \ref{conj:2}, let $N_p:=q_1\cdots q_k$ for some distinct primes $q_1,\ldots,q_k$ with $k\le 1.36\log p$. We take numbers of the form $2^a+1$ with an odd $a\in [5,p-2]$. We want to compute
$$
\left(\frac{2^a+1}{2^p-1}\right).
$$
This was done by Rotkiewicz in \cite{Ro}. Namely, write the Euclidean algorithm with even quotients and signed remainders:
\begin{eqnarray*}
p & = & (2k_1)a+\varepsilon_1 r_1,\quad \varepsilon_1\in \{\pm 1\},\quad 1\le r_1\le a-1\\
a & = & (2k_2) r_1+\varepsilon_2 r_2,\quad \varepsilon_2\in \{\pm 1\},\quad 1\le r_2\le r_1-1,\\
\ldots & = & \ldots\\
r_{\ell-2} & = & (2k_{\ell}) r_{\ell-1} +\varepsilon_{\ell} r_\ell,\quad \varepsilon_{\ell}\in \{\pm 1\},\quad r_{\ell}=1,
\end{eqnarray*}   
 where $\ell:=\ell(a,p)$ is minimal with $r_{\ell}=1$. Note that $r_i$ are all odd for $i=1,\ldots,\ell$. In particular, $r_j\ge 3$ for $j=1,\ldots,\ell-1$. Then
 $$
 \left(\frac{2^a+1}{2^p-1}\right)=\left(\frac{2^p-1}{2^a+1}\right)=\left(\frac{(2^a)^{2k_1}\cdot 2^{\varepsilon_1 r_1}-1}{2^a+1}\right)= \left(\frac{2^{\varepsilon_1 r_1}-1}{2^a+1}\right)=\left(\frac{2^{r_1}-1}{2^a+1}\right).
 $$
 The right--most equality is clear if $\varepsilon_1=1$, and if $\varepsilon_1=-1$, then 
 \begin{eqnarray*}
 \left(\frac{2^{-r_1}-1}{2^a+1}\right) & = & \left(\frac{2}{2^a+1}\right)^{r_1} \left(\frac{2^{-r_1}-1}{2^a+1}\right)=\left(\frac{2^{r_1}(2^{-r_1}-1)}{2^a+1}\right)\\
 & = & \left(\frac{1-2^{r_1}}{2^a+1}\right)=\left(\frac{-1}{2^a+1}\right)\left(\frac{2^{r_1}-1}{2^a+1}\right) = \left(\frac{2^{r_1}-1}{2^a+1}\right).
 \end{eqnarray*}
 The above calculation shows how to transit from the first step to the second step, or more generally from a step with $2^{r_j}-1$ in the bottom to a step with $2^{r_j}-1$ in the top. 
 For the next step, we write 
 $$
 \left(\frac{2^{r_1}-1}{2^a+1}\right)=\left(\frac{2^a+1}{2^{r_1}-1}\right)=\left(\frac{(2^{r_1})^{2k_2}\cdot 2^{\varepsilon_2 r_2}+1}{2^{r_1}-1}\right)=\left(\frac{2^{\varepsilon_2 r_2}+1}{2^{r_1}-1}\right)=\left(\frac{2^{r_2}+1}{2^{r_1}-1}\right).
 $$
 The last inequality above is clear if $\varepsilon_2=1$. If $\varepsilon_2=-1$, then since $r_1\ge 3$, we have 
 $$
 \left(\frac{2^{-r_2}+1}{2^{r_1}-1}\right)=\left(\frac{2}{2^{r_1}-1}\right)^{r_2}\left(\frac{2^{-r_2}+1}{2^{r_1}-1}\right)=\left(\frac{2^{r_2}(2^{-r_2}+1)}{2^{r_1}-1}\right)=\left(\frac{2^{r_2}+1}{2^{r_1}-1}\right).
 $$
 At step $\ell$ we end up with
 $$
 \left(\frac{2^{r_{\ell}}+(-1)^{\ell}}{2^{r_{\ell-1}}+(-1)^{\ell-1}}\right).
 $$
 If $\ell(a,p)$ is odd, we get
 $$
 \left(\frac{2^1-1}{2^{r_{\ell-1}}+1}\right)=1,
 $$
 and if $\ell(a,p)$ is even we get 
 $$
 \left(\frac{2^1+1}{2^{r_{\ell-1}}-1}\right)=-\left(\frac{2^{r_{\ell-1}}-1}{3}\right)=-\left(\frac{1}{3}\right)=-1.
 $$
 We thus get that 
 $$
 \left(\frac{2^a+1}{2^p-1}\right)=(-1)^{\ell+1}.
 $$
 We select the subset ${\mathcal A}(p)$ of odd $a$ in the interval $[5,p-2]$ such that $\ell\equiv 0\pmod 2$. We assume that there are a positive proportion of such, namely that there is a constant $c_1>0$ such that for large $p$, there are $>c_1p$ odd numbers  $a\in [5,p-2]$ such that $\ell(a,p)\equiv 0\pmod 2$.  So, we have 
 $$
\prod_{i=1}^k \left(\frac{2^a+1}{q_i}\right)=-1\quad {\text{\rm for}}\quad a\in {\mathcal  A}(p).
$$
We next conjecture that for such $a$, the values are 
\begin{equation}
\label{eq:2}
\left(\left(\frac{2^a+1}{q_i}\right),1\le i\le k\right)
\end{equation}
are uniformly distributed among the $2^k$ vectors ${\underbrace{(\pm 1,\pm 1,\cdots,\pm 1)}_{k~{\text{\rm times}}}}$. Indeed, if not, then somehow for some $a$ the value of $2^a+1$ of the Legendre symbol ${\displaystyle{\left(\frac{2^a+1}{p_i}\right)}}$ should be determined in  terms of the values of the same symbol for $b\le a-1$. This can happen for example if:  
\begin{itemize}
\item[(i)] $2^a+1$ is a square. This never happens for $a\ge 4$.
\item[(ii)] $2^a+1$ is multiplicatively dependent over $\{2^b+1: 0\le b\le a-1\}$. This does not happen for $a\ge 4$ because of the Carmichael's Primitive Divisor Theorem: $2^a+1$ has a prime factor $p_a$ which is primitive in the sense that $p_a$ does not divide $2^b+1$ for any $b\le a-1$.
\end{itemize}
Well, so we fix $i\in \{1,\ldots,k\}$ and search for $a_i$ such that
\begin{equation}
\label{eq:3}
\left(\frac{2^{a_i}+1}{q_i}\right)=(-1)^{\delta_{ij}},
\end{equation}
where $\delta_{ij}$ is the Kronecker symbol. That is, $2^{a_i}+1$ is a quadratic residue modulo $p_j$ for all $j\ne i$ but it is not a quadratic residue modulo $q_i$. 
Do we expect to find it? Well, let us see. Fix $i$ in $\{1,\ldots,k\}$. The probability that $2^{a_i}+1$ verifies the Legendre conditions given by \eqref{eq:3} is $1/2^k$ so it is $(1-1/2^k)$ that they are not satisfied. 
Note that since ${\displaystyle{\left(\frac{2^{a_i}+1}{N_p}\right)=-1}}$ we know that an odd number of the $p=p_j$'s satisfy that ${\displaystyle{\left(\frac{2^{a_i}+1}{p_j}\right)=1}}$.
So, if we assume that this is so for all possible $a_i$'s, and assuming that there events are independent, we get that the probability that this be so is
$$
\ll \left(1-\frac{1}{2^{k}}\right)^{c_1 p}<\left(1-\frac{1}{p^{1.36\log 2}}\right)^{c_1p}<\left(1-\frac{1}{p^{0.95}}\right)^{c_1 p}\ll \frac{1}{e^{c_1 p^{0.05}}}.
$$
In the above, we used that $k<1.36\log p$ and $1.36\cdot \log 2<0.95$. Of course, this is for $i$ fixed and now we sum up over $i$ from $1$ to $k$ introducing another logarithmic factor in the above count. Since the series 
$$
\sum_{p} \frac{\log p}{e^{c_1 p^{0.05}}}
$$
converges, so we expect that the above event does not occur when $p>p_0$. So, we have the following conjecture.

\begin{conj}
Assume Conjectures \ref{conj:1} and \ref{conj:2}. Write $2^p-1=q_1\ldots q_k$ for $p>p_0$ with prime factors $q_1<\cdots<q_k$ and $k<1.36\log p$.  Then for each $i=1,\ldots,k$, there exists an odd $a_i\in [5,p-2]$ such that equalities 
\eqref{eq:3} hold. 
\end{conj}

The rest of the proof is unconditional. We will show that there exist integers $x_i$ such that 
\begin{equation}
\label{eq:10}
(-1)=\prod_{i=1}^k (2^{a_i}+1)^{x_i}\pmod {2^p-1}.
\end{equation}
Write $q_i-1=:2^{\alpha_i} R_i$ for $i=1,\ldots,k$, where $R_i$ is odd. Let 
$$
R:={\text{\rm lcm}}[R_i:1\le i\le k]
$$
and write $x_i=y_ iR$ for $i=1,\ldots,k$. Let $\rho_i$ be a primitive root modulo $q_i$. Write
\begin{equation}
\label{eq:100}
2^{a_i}+1=\rho_j^{b_{ij}}\pmod {q_j}.
\end{equation}
Conditions \eqref{eq:3} show that $b_{ij}\equiv \delta_{ij}\pmod 2$. Equation \eqref{eq:10} holds if and only if it holds one prime $q_j$ at a time. Thus, we want
$$
\rho_j^{(q_j-1)/2}\equiv \rho_j^{R\sum_{i=1}^k y_i b_{ij}}\pmod {q_j},
$$
which holds provided that 
$$
\frac{(q_j-1)}{2}\equiv R\sum_{i=1}^k y_i b_{ij}\pmod {q_j-1}.
$$
This in turn is equivalent to 
$$
2^{\alpha_j-1} \equiv (R/R_j) \sum_{i=1}^k y_i b_{ij}\pmod {2^{\alpha_j}}.
$$
Since $R/R_j$ is odd, it follows that it is invertible modulo $2^{\alpha_j}$. Writing $(R/R_j)^*$ for the inverse of $(R/R_j)$ modulo $2^{\alpha_j}$, we get that 
$$
2^{\alpha_j-1} (R/R_j)^*\equiv \sum_{i=1}^k y_i b_{ij}\pmod {2^{\alpha_j}}.
$$
Since $(R/R_j)^*$ is odd the left--hand side is just $2^{\alpha_j-1}\pmod {2^{\alpha_j}}$. Thus, 
$$
2^{\alpha_j-1} \equiv \sum_{i=1}^k y_i b_{ij}\pmod {2^{\alpha_j}}.
$$
This is a linear system of modular equations for $i=1,\ldots,k$. To see that it is nondegenerate note that the coefficient matrix ${\mathcal B}=(b_{ij})_{1\le i,j\le k}$ modulo $2$ is in fact the identity matrix. Hence, its determinant is odd integer, so invertible modulo powers of $2$, which shows that there exist an integer solution $y_1,\ldots,y_k$. To solve it, we can generate for each $i,j$ the number $b_{i,j}\pmod {2^{\alpha_j}}$ appearing in \eqref{eq:100} as an integer in the interval $[0,2^{\alpha_j}-1]$. Having done that, we solve the linear system 
$$
\sum_{i=1}^k y_i b_{ij}=2^{\alpha_j-1}\quad {\text{\rm for}}\quad j=1,2,\ldots,k.
$$
This is non-degenerate since the determinant of the coefficient matrix is odd. Thus, $(y_1,\ldots,y_k)$ are some rational numbers. Now we treat them as residue classes modulo $2^{\alpha}$, where $\alpha:=\max\{\alpha_i: 1\le i\le k\}$
(by inverting the odd determinant modulo $2^{\alpha}$). 
These ones are the $y_i$'s that we are looking for. 

We  implemented this and checked it for all primes $p\le 250$. We present our approach in Algorithm~\ref{algo1}. 

 \begin{algorithm}[!htb]
	\caption{} 
        %\captionlistentry{}
        \label{algo1}
	\DontPrintSemicolon
	%\KwIn{A positive integer $B$}	
	%\KwOut{ Extended subspace trail $(V_0^{(j)}, V_1^{(j)}, \ldots, V_{l-1}^{(j)})$ and length $l$}	
        
        \For{Prime $p \leq 250$}
         { 
         
        Factor $2^p-1=q_1\cdots  q_k$, where $q_i$ is prime for $1 \leq i \leq k$;
        
         \For{$j=1$ to $k$}{
                    Find odd $a_j \in [5, p-2]$ such that the
                    Legendre symbol $\left(\frac{2^{a_j}+1}{q_i}\right)=(-1)^{\delta_{ij}}$ where $\delta_{ij}$ is the Kronecker symbol.

         }

         Take  a primitive root $\rho_i$ modulo $q_i$ for $1 \leq i \leq k$;

         Find $b_{ij}$ such that $2^{a_i}+1=\rho_j^{b_{ij}}\pmod {q_j}$ for $1 \leq i,j \leq k$;

         Find largest $\alpha_i$ such that  $2^{\alpha_i}$ is a divisior of $q_i-1$ for  $1 \leq i \leq k$;

         Calculate $\alpha=\max\{\alpha_i: 1\le i\le k\}$;

         Solve the system of linear equations  
$
\sum_{i=1}^k y_i b_{ij}=2^{\alpha_j-1}\quad {\text{\rm for}}\quad j=1,2,\ldots,k.
$ in $\mathbb{Z}_{\alpha}$

         }
            
	%\uIf{the number of nonzero coordinates $x$ after the cube sum  satisfies $x\leq \mathcal{T}$,    }
	 %{Output  \ascon}
	 %\Else{Output random source}
	
\end{algorithm} 

The factorization of $2^p-1$ is known for all primes $p<1000$. Surely, we can use the same algorithm modulo $2^n-1$ 
for $n\le 250$ and odd. Note that if 
$$
2^n-1=\prod_{j=1}^k q_j^{\alpha_j},
$$
then we only want to find a relation of the form
\begin{equation}
\label{eq:M}
(-1)\equiv \prod_{i=1}^k (2^{a_i}+1)^{x_i}\pmod {q_1\ldots q_k}.
\end{equation}
Indeed, if we have found the above relation, then 
$$
q_1\ldots q_k\mid (2^{a_1}+1)^{x_1}\cdots (2^{a_k}+1)^{x_k}+1.
$$
Writing $Q:=(2^n-1)/(q_1\ldots q_k)$, we then get easily that 
$$
(2^n-1)\mid (2^{a_1}+1)^{x_1 Q}   (2^{a_2}+1)^{x_2 Q}\cdots (2^{a_k}+1)^{x_k Q}+1.
$$
Thus, 
$$
(-1)\equiv \prod_{i=1}^k (2^{a_i}+1)^{x_i Q}\pmod {2^n-1}.
$$
Thus, we factor $2^n-1$, take $q_1,\ldots,q_k$ to be all its distinct prime factors and attempt to find some numbers $a_1,\ldots,a_k$ in $[5,n-2]$ such that the congruences~\eqref{eq:3} are satisfied. If we are successful, 
then the argument based on the matrix with odd determinant will work to find a solution of \eqref{eq:M}, which in turn can be easily lifted to a solution modulo $2^n-1$.  
The factorizations of $2^n-2$  with weight 2 factors for odd $33 \leq n \leq 249$ are given in Table~\ref{tab:1}. 

\begin{remark}
We have checked that Algorithm~\textup{\ref{algo1}} works for most primes up to $250$. But there are a few primes like $47$ for which there is no $a_j \in [5, p-2]$ such that $\left(\frac{2^{a_j}+1}{q_i}\right)=(-1)^{\delta_{ij}}$. In these cases, we use the following trick. 
We first find $a_i$ and calculate  $(\frac{2^{a_j}+1}{q_i})=(-1)^{d_{i,j}}$.  
Ideally, $d_{i,j}$ should be Kronecker symbols, but if they are not, we can just record what they are. Because $d_{i,j}$ are no longer Kronecker symbols, we cannot be certain that the system is solvable because it may have an even determinant and we cannot invert the matrix modulo powers of $2$. However, we checked primes (and odd integers) up to $250$. We observed that in the case of failure, we can use this trick and always get suitable $a_i$'s such that the corresponding matrix has odd determinant, and is therefore invertible.
\end{remark}

 \setstretch{1.5}
{\tiny
\setlength\tabcolsep{1.5pt}
\begin{longtable}{|c|l|}
\caption{Factorization of $2^n-2 \pmod {2^n-1}$ for odd $33 \leq n \leq 249$.} \label{tab:1} \\
\hline
$n=33$ & $(2^5+1)^{599478} \cdot (2^{13}+1)^{299739} \cdot (2^{29}+1)^{1798434}$ \\
\hline 
 $n=35$ &  $\left((2+1)(2^{17}+1)\right)^{967995}\cdot  (2^{29}+1)^{276570}$ \\
 \hline
$n=37$  & $(2^5+1)^{77039772} \cdot (2^{13}+1)^{19259943}$\\ 
 \hline 
 $n=39$  & $\left((2^{11}+1) (2^{21}+1)\right)^{1592955}$\\ 
 \hline 
 $n=41$  & $(2^{9}+1)^{20111512782} \cdot (2^{13}+1)^{3351918797}$\\ 
 \hline 
$n=43$  & $\left((2^{5}+1) (2^{17}+1) (2^{23}+1)\right)^{593211015}$\\ 
 \hline 
 $n=45$  & $(2+1)^{407925} \cdot (2^{13}+1)^{349650} \cdot \left((2^{25}+1) (2^{33}+1) (2^{41} +1)\right)^{116550}$\\ 
 \hline 
 $n=47$ & $(2^{11}+1)^{1927501725} \cdot (2^{37}+1)^{435242325} \cdot (2^{41}+1)^{ 1616614350}$ \\
 \hline 
 $n=49$ & $ (2^9+1)^{34630287489} \cdot (2^{11}+1)^{3393768173922}$ \\
 \hline 
  $n=51$ & $ (1+2^{ 29 })^{ 150009615 }$ \\
 \hline 
 $n=53$ & $(1+2^{ 5 })^{ 6512186850 } \cdot (1+2^{ 15 })^{ 3506562150 } \cdot (1+2^{ 21 })^{ 250468725 }$ \\
 \hline 
 \multirow{ 2}{*}{$n=55$}  & $(1+2)^{ 6588945 } \cdot (1+2^{ 11 })^{ 5856840 } \cdot (1+2^{ 17 })^{ 732105 }\cdot$\\ 
  & $ (1+2^{ 25 })^{ 1464210 } \cdot (1+2^{ 33 })^{ 10249470 } \cdot (1+2^{ 47 })^{ 732105 }$ \\
 \hline 
 \multirow{ 1}{*}{$n=57$}   & $(1+2^{ 5 })^{ 396029391534 } \cdot (1+2^{ 17 })^{ 1188088174602 } \cdot (1+2^{ 21 })^{ 594044087301 }\cdot (1+2^{ 47 })^{ 198014695767 }$\\
 \hline 
 $n=59$   & $(1+2^{ 7 })^{ 3663925098759300 } \cdot (1+2^{ 13 })^{ 305327091563275 }$ \\
 \hline 
 $n=61$   & $(1+2^{ 9 })^{ 1152921504606846975 }$\\
 \hline 
   \multirow{ 1}{*}{$n=63$ }  & $(1+2)^{ 42958503 } \cdot (1+2^{ 5 })^{ 3735522 } \cdot (1+2^{ 39 })^{ 56032830 }\cdot (1+2^{ 43 })^{ 44826264 } \cdot (1+2^{ 47 })^{ 29884176 }$ \\
   \hline 
   $n=65$ & $(1+2^{ 17 })^{ 72647571779055 } \cdot (1+2^{ 23 })^{ 72647571779055 } \cdot (1+2^{ 29 })^{ 72647571779055 }$\\
   \hline 
   $n=67$   & $(1+2^{ 5 })^{ 15295807610659665 }$\\
   \hline 
   \multirow{ 1}{*}{$n=69$}   & $(1+2^{ 11 })^{ 36566619637113225 } \cdot (1+2^{ 17 })^{ 2437774642474215 } \cdot (1+2^{ 53 })^{ 19502197139793720 } \cdot (1+2^{ 67 })^{ 21939971782267935 }$\\
    \hline 
    $n=71$   & $(1+2^{ 11 })^{ 3659326099961865 } \cdot (1+2^{ 13 })^{ 14637304399847460 }$\\
    \hline 
    $n=73$   & $(1+2^{ 31 })^{ 1726845200475585 } \cdot (1+2^{ 45 })^{ 107064402429486270 }$ \\
    \hline 
    \multirow{ 2}{*}{$n=75$} & $ (1+2)^{ 36654975 } \cdot (1+2^{ 39 })^{ 17832150 } \cdot (1+2^{ 41 })^{ 9906750 } \cdot $\\ 
    & $(1+2^{ 43 })^{ 7925400 } \cdot (1+2^{ 53 })^{ 57459150 } \cdot (1+2^{ 55 })^{ 15850800 } \cdot (1+2^{ 63 })^{ 43589700 }$ \\
    \hline 
    \multirow{ 2}{*}{$n=77$ }  & $(1+2^{ 25 })^{ 290641821624556479 } \cdot (1+2^{ 31 })^{ 290641821624556479 } \cdot$ \\
    & $(1+2^{ 41 })^{ 290641821624556479 } \cdot (1+2^{ 67 })^{ 581283643249112958 }$ \\
    \hline 
 $ n=79$   & $(1+2^{ 9 })^{ 12102186118644337359 } \cdot (1+2^{ 15 })^{ 12102186118644337359 } \cdot (1+2^{ 41 })^{ 12102186118644337359 }$ \\
             \hline
\multirow{ 2}{*}{$n=81$}   & $(1+2)^{ 106331083505919 } \cdot (1+2^{ 25 })^{ 155626336778778 } \cdot (1+2^{ 37 })^{ 105108887143782 } \cdot$\\
   & $ (1+2^{ 39 })^{ 155626336778778 } \cdot (1+2^{ 43 })^{ 4073987873790 } $ \\
\hline 
$n=83$   & $(1+2^{ 11 })^{ 7239076764159456135965 }$ \\
\hline 
\multirow{ 1}{*}{$n=85$}   & $(1+2^{ 9 })^{ 4760486403166879215 } \cdot (1+2^{ 13 })^{ 4760486403166879215 } \cdot  (1+2^{ 23 })^{ 4760486403166879215 }$ \\
\hline 
\multirow{ 2}{*}{$n=87$}   & $(1+2^{ 39 })^{ 3371346107168004 } \cdot (1+2^{ 41 })^{ 280945508930667 } \cdot (1+2^{ 53 })^{ 2809455089306670 } \cdot$\\
& $ (1+2^{ 61 })^{ 4214182633960005 } \cdot (1+2^{ 71 })^{ 1685673053584002 } \cdot (1+2^{ 83 })^{ 280945508930667 }$ \\
\hline 
$n=89$   & $(1+2^{ 13 })^{ 309485009821345068724781055 }$ \\
\hline 
\multirow{ 2}{*}{$n=91$}   & $(1+2^{ 59 })^{ 280368506850705 } \cdot (1+2^{ 67 })^{ 1682211041104230 } \cdot$\\
& $ (1+2^{ 71 })^{ 280368506850705 } \cdot  (1+2^{ 73 })^{ 280368506850705 } \cdot (1+2^{ 81 })^{ 3364422082208460 }$ \\
\hline 
$n=93$   & $(1+2^{ 17 })^{ 2305843010287435773 }$ \\
\hline 
$n=95$   & $(1+2^{ 43 })^{ 7354378117756963125 } \cdot (1+2^{ 51 })^{ 7354378117756963125 }$ \\
\hline 
$n=97$   & $(1+2^{ 5 })^{ 612535370185410489825162846 } \cdot (1+2^{ 9 })^{ 102089228364235081637527141 }$\\
\hline
\multirow{ 2}{*}{$n=99$}   & $(1+2)^{ 160190876329840719 } \cdot (1+2^{ 23 })^{ 160190876329840719 } \cdot (1+2^{ 35 })^{ 58251227756305716 }\cdot$\\ 
& $  (1+2^{ 57 })^{ 29125613878152858 } \cdot (1+2^{ 59 })^{ 101939648573535003 } \cdot  (1+2^{ 75 })^{ 58251227756305716 }$ \\
\hline 
$n= 101 $&$(1+2^{7})^{261479084205457681314981849}\cdot(1+2^{9})^{1045916336821830725259927396}$\\
\hline
$n= 103 $ & $(1+2^{5})^{8204858250687037849538541156}\cdot(1+2^{9})^{2051214562671759462384635289}$ \\
\hline
\multirow{ 3}{*}{$n= 105 $} & $(1+2^{7})^{736412106675}\cdot(1+2^{29})^{6627708960075}\cdot(1+2^{37})^{1472824213350}\cdot$ \\
           & $(1+2^{55})^{6627708960075}\cdot(1+2^{69})^{15464654240175}\cdot(1+2^{79})^{736412106675}\cdot$ \\
           & $(1+2^{83})^{4418472640050}\cdot(1+2^{85})^{441847264005}\cdot(1+2^{87})^{13255417920150}$ \\
\hline
$n= 107 $ & $(1+2^{5})^{27043212804868893898596335048021}$ \\
\hline
\multirow{ 1}{*}{$n= 109 $} & $(1+2^{7})^{744308608310570490215126499806}\cdot (1+2^{15})^{372154304155285245107563249903}$ \\
\hline
\multirow{ 2}{*}{$n= 111 $} & $(1+2^{17})^{2078233794395472907116}\cdot(1+2^{31})^{742226355141240323970}\cdot (1+2^{39})^{890671626169488388764}$ \\
           & $ (1+2^{71})^{180254971962872650107}\cdot (1+2^{87})^{519558448598868226779}$ \\
\hline
\multirow{ 3}{*}{$n= 113 $} & $(1+2^{15})^{82901226266607482846190}\cdot(1+2^{25})^{13816871044434580474365}\cdot$ \\
           & $(1+2^{29})^{37854441217628987601}\cdot(1+2^{75})^{13816871044434580474365}\cdot$ \\
           & $(1+2^{97})^{82901226266607482846190}$ \\
\hline
\multirow{ 3}{*}{$n= 115 $}&$(1+2^{17})^{23588654041464621525}\cdot(1+2^{23})^{165120578290252350675}\cdot$\\ &$ (1+2^{39})^{23588654041464621525}\cdot(1+2^{45})^{23588654041464621525}\cdot$\\ &$ (1+2^{75})^{188709232331716972200}$\\
\hline
\multirow{ 3}{*}{$n= 117 $} & $(1+2^{5})^{350280341971560}\cdot(1+2^{11})^{481635470210895}\cdot(1+2^{31})^{1225981196900460}\cdot$ \\
           & $(1+2^{55})^{1269766239646905}\cdot(1+2^{71})^{1225981196900460}\cdot(1+2^{87})^{744345726689565}\cdot$ \\
           & $(1+2^{93})^{1903697510715}\cdot(1+2^{111})^{1094626068661125}\cdot(1+2^{115})^{1182196154154015}$ \\
\hline
\multirow{ 3}{*}{$n= 119 $ }& $(1+2^{21})^{121807344007626864485535}\cdot(1+2^{25})^{28109387078683122573585}\cdot$ \\
           & $(1+2^{51})^{6635419517925198843570}\cdot(1+2^{81})^{5968559856373716359791215}\cdot$ \\
           & $(1+2^{97})^{852651408053388051398745}\cdot(1+2^{109})^{6635419517925198843570}$ \\
\hline
\newpage
\hline
\multirow{ 2}{*}{$n= 121 $} & $(1+2^{9})^{99244104353509123769903900571}\cdot(1+2^{19})^{893196939181582113929135105139}\cdot$ \\
           & $(1+2^{25})^{893196939181582113929135105139}\cdot(1+2^{43})^{1786393878363164227858270210278}$ \\
\hline
\multirow{ 3}{*}{$n= 123 $} & $(1+2^{5})^{38263506571610465341512132024}\cdot(1+2^{9})^{19131753285805232670756066012}\cdot$ \\
& $(1+2^{27})^{4782938321451308167689016503}\cdot (1+2^{53})^{28697629928707849006134099018}\cdot$\\
&$ (1+2^{113})^{7726284980805959347805334351}$ \\
\hline
\multirow{ 3}{*}{$n= 125 $} & $(1+2^{23})^{2898591397871459238374625}\cdot(1+2^{29})^{644131421749213164083250}\cdot$\\
& $(1+2^{95})^{1610328554373032910208125}$  $(1+2^{109})^{4186854241369885566541125}\cdot$\\
& $(1+2^{121})^{1932394265247639492249750}$ \\
\hline
$n= 127 $ & $(1+2^{5})^{28356863910078205288614550619314017621}$ \\
\hline
$n= 129 $ & $(1+2^{9})^{8471295533565243108183419405055}\cdot(1+2^{79})^{9529016348217371325290685495}$ \\
\hline
\multirow{ 1}{*}{$n= 131 $} & $(1+2^{9})^{1293849303881895298339827404683529321}\cdot  (1+2^{15})^{2587698607763790596679654809367058642}$ \\
\hline
\multirow{ 2}{*}{$n= 133 $} & $(1+2^{5})^{27256203475454233141905720953493}\cdot (1+2^{17})^{81768610426362699425717162860479}$ \\
& $(1+2^{45})^{81768610426362699425717162860479}$ \\
\hline
 
\multirow{ 4}{*}{$n= 135 $} & $(1+2^{9})^{1390256215369200900}\cdot(1+2^{25})^{9826330930229511961200}\cdot(1+2^{47})^{38551429107264868200}$ \\
& $(1+2^{55})^{2813183451799578021150}\cdot(1+2^{89})^{7686726614776311776100}$ \\
& $(1+2^{101})^{1406591725899789010575}\cdot(1+2^{109})^{4576375896941567062575}$ \\
& $(1+2^{117})^{1042692161526900675}\cdot(1+2^{121})^{2119793164384189072275}$ \\
\hline
\multirow{ 1}{*}{$n= 137 $}&$(1+2^{5})^{39741006355730039527321333167397040041}\cdot (1+2^{7})^{79482012711460079054642666334794080082}$ \\
\hline
\multirow{ 1}{*}{$n= 139 $}&$(1+2^{9})^{17408530362059304982034022473992637175}\cdot  (1+2^{17})^{457116770994992690994328817697837300}$ \\
\hline
\multirow{ 2}{*}{$n= 141 $}&$(1+2)^{1216799735702178355508978464575}\cdot(1+2^{51})^{110618157791107123228088951325}\cdot$\\&$ (1+2^{65})^{63210375880632641844622257900}\cdot(1+2^{121})^{56889338292569377660160032110}$ \\
\hline
\multirow{ 3}{*}{$n= 145 $}&$(1+2^{41})^{534639083977880631530660485081925505}\cdot(1+2^{49})^{76377011996840090218665783583132215}\cdot$\\
&$(1+2^{135})^{305508047987360360874663134332528860}\cdot(1+2^{137})^{4144489023084345980857833217689345}\cdot$\\
&$(1+2^{139})^{534639083977880631530660485081925505}$ \\
\hline

\multirow{ 3}{*}{$n= 147 $}&$(1+2^{15})^{17249119260282613026137951811234}\cdot(1+2^{59})^{51747357780847839078413855433702}\cdot$\\
&$(1+2^{67})^{43122798150706532565344879528085}\cdot(1+2^{71})^{96786724738252439757774062940813}\cdot$\\
&$(1+2^{73})^{8624559630141306513068975905617}$ \\
\hline
$n= 149 $&$(1+2^{9})^{29933886172524326364132038117944134026225}$ \\
\hline
%\newpage
%\hline
\multirow{ 3}{*}{$n= 151 $}&$(1+2^{35})^{47657859344287051433215338407025}\cdot(1+2^{53})^{47657859344287051433215338407025}\cdot$\\
&$(1+2^{55})^{95315718688574102866430676814050}\cdot(1+2^{81})^{95315718688574102866430676814050}\cdot$\\
&$(1+2^{119})^{47657859344287051433215338407025}$ \\
\hline 
\multirow{ 3}{*}{$n= 153 $}&$(1+2^{67})^{74105228687928761744692516074690}\cdot(1+2^{85})^{566868663181345103125787385}\cdot$\\
&$(1+2^{101})^{185263071719821904361731290186725}\cdot(1+2^{115})^{24701742895976253914897505358230}\cdot$\\
&$(1+2^{125})^{66694705819135885570223264467221}\cdot(1+2^{147})^{22231568606378628523407754822407}$ \\
\hline
\multirow{ 4}{*}{$n= 155 $}&$(1+2^{17})^{62671642336461616797239779725}\cdot(1+2^{43})^{35812367049406638169851302700}\cdot$\\
&$(1+2^{51})^{17906183524703319084925651350}\cdot(1+2^{59})^{2984363920783886514154275225}\cdot$\\
&$(1+2^{73})^{26859275287054978627388477025}\cdot(1+2^{101})^{17906183524703319084925651350}\cdot$\\
&$(1+2^{123})^{865014224607504542831738415625}$ \\
\hline
\multirow{ 2}{*}{$n= 157 $}&$(1+2^{15})^{1707444675887681902216221662393643900}\cdot (1+2^{17})^{341488935177536380443244332478728780}\cdot$\\
&$(1+2^{29})^{3841750520747284279986498740385698775}\cdot (1+2^{45})^{30734004165978274239891989923085590200}$ \\
\hline
\multirow{ 4}{*}{$n= 159 $}&$(1+2^{57})^{9362516203257056384802075}\cdot(1+2^{65})^{44348760962796582875378250}\cdot$\\
&$(1+2^{69})^{4434876096279658287537825}\cdot(1+2^{89})^{38435592834423705158661150}\cdot$\\
&$(1+2^{101})^{79907677410444293469150}\cdot(1+2^{123})^{5913168128372877716717100}\cdot$\\
&$(1+2^{127})^{2002242486366485713350}\cdot(1+2^{137})^{48783637059076241162916075}$ \\
\hline
 
\multirow{ 3}{*}{$n= 161 $}&$(1+2^{29})^{93343471924356402246389002034385}\cdot(1+2^{43})^{62228981282904268164259334689590}\cdot$\\
&$(1+2^{47})^{82971975043872357552345779586120}\cdot(1+2^{87})^{217801434490164938574907671413565}\cdot$\\
&$(1+2^{157})^{248915925131617072657037338758360}$ \\
\hline
\newpage
\hline
\multirow{ 2}{*}{$n= 163 $}&$(1+2^{13})^{168486137937535997136381884224759350}\cdot(1+2^{87})^{5616204597917866571212729474158645}\cdot$\\
&$(1+2^{89})^{1925555862143268538701507248282964}\cdot(1+2^{119})^{78626864370850131996978212638221030}$ \\
\hline
\multirow{ 5}{*}{$n= 165 $}&$(1+2^{61})^{39948352158132627380823842541450}\cdot(1+2^{105})^{9321282170230946388858896593005}\cdot$\\
&$(1+2^{109})^{13316117386044209126941280847150}\cdot(1+2^{113})^{15535470283718243981431494321675}\cdot$\\
&$(1+2^{119})^{19974176079066313690411921270725}\cdot(1+2^{123})^{79896704316265254761647685082900}\cdot$\\
&$(1+2^{127})^{3698921496123391424150355790875}\cdot(1+2^{135})^{5326446954417683650776512338860}\cdot$\\
&$(1+2^{147})^{4438705795348069708980426949050}\cdot(1+2^{157})^{13316117386044209126941280847150}$ \\
\hline 
\multirow{ 2}{*}{$n= 167 $}&$(1+2^{5})^{350060123390813635242448130256489390771914057740}\cdot$\\
&$(1+2^{33})^{17503006169540681762122406512824469538595702887}$ \\
\hline
\multirow{ 3}{*}{$n= 169 $}&$(1+2^{15})^{7089726406583596958466287242575266870940}\cdot$\\
&  $(1+2^{25})^{7286663251210919096201461888202357617355}\cdot$\\
& $(1+2^{55})^{29146653004843676384805847552809430469420}$\\
\hline
\multirow{ 4}{*}{$n= 171 $}&$(1+2^{49})^{49149123355767553400835304429946193135}\cdot$ $(1+2^{55})^{9829824671153510680167060885989238627}\cdot$\\
&$(1+2^{77})^{9829824671153510680167060885989238627}\cdot$ $ (1+2^{99})^{9829824671153510680167060885989238627}\cdot$\\
&$(1+2^{109})^{9829824671153510680167060885989238627}\cdot$ $(1+2^{159})^{9829824671153510680167060885989238627}\cdot$\\
&$(1+2^{163})^{9829824671153510680167060885989238627}$ \\
\hline
\multirow{ 2}{*}{$n= 173 $}&$(1+2^{21})^{752712011377013221558430642567508556008861}\cdot (1+2^{61})^{752712011377013221558430642567508556008861}\cdot$\\
&$(1+2^{69})^{752712011377013221558430642567508556008861}\cdot (1+2^{77})^{31613904477834555305454086987835359352372162}$ \\
\hline
\multirow{ 5}{*}{$n= 175 $}&$(1+2^{19})^{281198684623467689204913686779425}\cdot$ $(1+2^{43})^{540766701198976325394064782268125}\cdot$\\
&$(1+2^{61})^{12329480787336660218984677035713250}\cdot$ $(1+2^{97})^{129784008287754318094575547744350}\cdot$\\
&$(1+2^{105})^{6651799969190141587990774087125}\cdot$ $(1+2^{123})^{11139794044698912303117734514723375}\cdot$\\
&$(1+2^{141})^{4975053651030582193625395996866750}\cdot$ $(1+2^{163})^{1838606784076519506339820259711625}\cdot$\\
&$(1+2^{165})^{1114982889070054279163020169625}$ \\
\hline

\multirow{ 3}{*}{$n= 177 $}&$(1+2^{19})^{102104448867391604403906908898445293975}\cdot (1+2^{57})^{4288386852430447384964090173734702346950}\cdot$\\
        &$(1+2^{93})^{81683559093913283523125527118756235180}\cdot (1+2^{103})^{1286516055729134215489227052120410704085}\cdot$\\
        &$(1+2^{133})^{1837880079613048879270324360172015291550}\cdot (1+2^{163})^{568867643689753224536052778148480923575}$ \\
\hline
\multirow{ 2}{*}{$n= 179 $}&$(1+2^{21})^{1713334865061395551905989490523888483510297545}\cdot$ \\
&$(1+2^{29})^{17731656063809998410201669013040877638868476180}$ \\
\hline
 
\multirow{ 3}{*}{$n= 181 $} & $(1+2^{11})^{16825262628616094460214312446168738419261880}\cdot$\\
&  $(1+2^{21})^{21031578285770118075267890557710923024077350}\cdot$ \\
          & $(1+2^{31})^{2103157828577011807526789055771092302407735}\cdot$ \\
\hline

\multirow{ 3}{*}{$n= 183 $} & $(1+2^{17})^{5301342361191869603932356617428842175355175}\cdot$\\
         & $(1+2^{75})^{171011043909415148513946987658994908882425}\cdot$ \\
          & $(1+2^{177})^{1279634363046313352673327459379375697499525}$ \\
\hline
\multirow{ 5}{*}{$n= 185 $} & $(1+2^{21})^{24289606148175875174394125504156277451293856980}\cdot$\\ 
          &$(1+2^{25})^{56675747679077042073586292843031314053018999620}\cdot$ \\
          & $(1+2^{61})^{6072401537043968793598531376039069362823464245}\cdot$\\
         & $(1+2^{73})^{28337873839538521036793146421515657026509499810}\cdot$ \\
          & $(1+2^{121})^{3333867510533943651387428990766547885471705860}$ \\
\hline
\multirow{ 4}{*}{$n= 187 $} & $(1+2^{9})^{52332852605905914814562661586433935229041537035}\cdot$\\
        & $(1+2^{19})^{52332852605905914814562661586433935229041537035}\cdot$ \\
          & $(1+2^{41})^{203629776676676711340710745472505584548799755}\cdot$\\
         & $ (1+2^{89})^{104665705211811829629125323172867870458083074070}$ \\
\hline
\multirow{ 5}{*}{$n= 189 $} & $(1+2^{7})^{954808575327093401067436576289941581}\cdot (1+2^{11})^{76741427691674298191288254054776183774}\cdot$ \\
          & $(1+2^{21})^{862835224381042468504831900414614}\cdot (1+2^{33})^{8687708795283882814108104232616171748}\cdot$ \\
& $(1+2^{39})^{6515781596462912110581078174462128811}\cdot (1+2^{73})^{60813961566987179698756729628313202236}\cdot$ \\
& $(1+2^{89})^{53091553748957061641771748088209938460}\cdot (1+2^{125})^{62503238277181268023722194340210791187}\cdot$ \\
& $(1+2^{157})^{121386597889660918208232678583498177479}$ \\
\hline
\newpage
\hline
\multirow{ 5}{*}{$n= 191 $} & $(1+2^{55})^{1025444869877616060103489665965652402208725}\cdot$\\ 
          & $(1+2^{77})^{1860992541629747664632259023419146952156575}\cdot$ \\
& $(1+2^{179})^{1063424309502712951218433727668083972660900}\cdot$\\
  & $ (1+2^{183})^{2126848619005425902436867455336167945321800}\cdot$ \\
& $(1+2^{189})^{138245160235352683658396384596850916445917}$ \\
\hline
\multirow{ 3}{*}{$n= 193 $} & $(1+2^{5})^{690644713229389686815238348164730529976649425988651}\cdot$\\
      & $(1+2^{9})^{76738301469932187423915372018303392219627713998739}\cdot$ \\
& $(1+2^{13})^{98663530461341383830748335452104361425235632284093}$ \\
\hline
\multirow{ 3}{*}{$n= 195 $} & $(1+2^{31})^{397218618589975176651322156679935564968150}\cdot$\\
   & $ (1+2^{129})^{42129247426209488432715986314538620526925}\cdot$ \\
& $(1+2^{163})^{66203103098329196108553692779989260828025}$ \\
\hline
\multirow{ 2}{*}{$n= 197 $} & $(1+2^{5})^{47788807121282329843547481918370106279929779662675299482}\cdot $\\
& $(1+2^{13})^{7964801186880388307257913653061684379988296610445883247}$ \\
\hline
\multirow{ 2}{*}{$n= 199 $} & $(1+2^{7})^{9012349943070385113930109307619809450853768536749084363}\cdot $\\
& $(1+2^{27})^{234321098519830012962182841998115045722197981955476193438}$ \\
\hline
\multirow{ 4}{*}{$n= 201 $} & $(1+2^{67})^{201}\cdot (1+2^{107})^{915031412652462933978517192307754031960084170}\cdot$ \\
& $(1+2^{119})^{14182986896113175476667016480770187495381304635}\cdot$\\
 & $(1+2^{145})^{3507620415167774580250982570513057122513655985}\cdot$ \\
& $(1+2^{177})^{83184673877496630361683381118886730178189470}$ \\
\hline
\multirow{ 7}{*}{$n= 203 $} & $(1+2^{57})^{718674251279934430341052428990271148449812}\cdot$\\
   & $(1+2^{81})^{51205540403695328161799985565556819327049105}\cdot$ \\
& $(1+2^{111})^{6288399698699426265484208753664872548935855}\cdot$\\
& $(1+2^{127})^{18865199096098278796452626260994617646807565}\cdot$ \\
& $(1+2^{153})^{2695028442299754113778946608713516806686795}\cdot$\\
&$ (1+2^{175})^{64680682615194098730694718609124403360483080}\cdot$ \\
& $(1+2^{193})^{8085085326899262341336839826140550420060385}$ \\
\hline
 
\multirow{ 6}{*}{$n= 205 $}&$(1+2^{43})^{9467961424350347777980448419013907428323430550}\cdot$\\
 & $(1+2^{67})^{29756450190815378730795695031186566203302210300}\cdot$\\
&$(1+2^{131})^{676282958882167698427174887072421959165959325}\cdot$\\
&$(1+2^{145})^{7213684894743122116556532128772500897770232800}\cdot$\\
& $ (1+2^{157})^{772894810151048798202485585225625096189667800}\cdot$\\
& $(1+2^{187})^{16907073972054192460679372176810548979148983125}$\\

\hline
\multirow{ 6}{*}{$n= 207 $} & $(1+2^{5})^{33192619261066535128289058132930982761836775}\cdot$\\
&$(1+2^{121})^{46469666965493149179604681386103375866571485}\cdot$ \\
& $(1+2^{127})^{102131136187897031163966332716710716190267}\cdot$\\
& $(1+2^{179})^{19915571556639921076973434879758589657102065}\cdot$ \\
& $(1+2^{187})^{39831143113279842153946869759517179314204130}\cdot$\\
& $(1+2^{199})^{3983114311327984215394686975951717931420413}$ \\
\hline

\multirow{ 4}{*}{$n= 209 $} & $(1+2^{59})^{270250337042154559392793016292662090505143401152585}\cdot$\\
&$ (1+2^{61})^{360333782722872745857057355056882787340191201536780}\cdot$ \\
& $(1+2^{87})^{180166891361436372928528677528441393670095600768390}\cdot$\\
&$ (1+2^{97})^{25738127337348053275504096789777341952870800109770}$ \\
\hline
\multirow{ 3}{*}{$n= 211 $} & $(1+2^{9})^{1154929871973565275224108161553051512738585560465529713690}\cdot$\\&
$ (1+2^{13})^{64162770665198070845783786752947306263254753359196095205}\cdot$ \\
& $(1+2^{55})^{2309859743947130550448216323106103025477171120931059427380}$ \\
\hline
 
\multirow{ 6}{*}{$n= 213 $} & $(1+2^{11})^{5497736589318700986099894638477660690688115321335}\cdot$\\
 & $(1+2^{41})^{140967604854325666310253708678914376684310649265}\cdot$ \\
& $(1+2^{85})^{1832578863106233662033298212825886896896038440445}\cdot$\\
& $(1+2^{119})^{2356172823993728994042811987918997438866335137715}\cdot$ \\
& $(1+2^{161})^{1099547317863740197219978927695532138137623064267}\cdot$\\
& $(1+2^{203})^{1570781882662485996028541325279331625910890091810}$ \\
\hline 

\multirow{ 5}{*}{$n= 215 $} & $(1+2^{41})^{66443114885376278757699109185773253174500402025}\cdot$\\
& $(1+2^{97})^{31006786946508930086926250953360851481433520945}\cdot$ \\
& $(1+2^{119})^{186040721679053580521557505720165108888601125670}\cdot$\\
& $(1+2^{151})^{186040721679053580521557505720165108888601125670}\cdot$ \\
& $(1+2^{203})^{279061082518580370782336258580247663332901688505}$ \\
\hline
\multirow{ 5}{*}{$n= 217 $} & $(1+2^{13})^{1126940817087752014462533930787563178937611179075}\cdot$\\
& $(1+2^{57})^{125215646343083557162503770087507019881956797675}\cdot$ \\
& $(1+2^{89})^{50086258537233422865001508035002807952782719070}\cdot$\\
& $(1+2^{185})^{751293878058501342975022620525042119291740786050}\cdot$ \\
& $(1+2^{215})^{9537309672323317621714046370301517484866488275}$ \\
\hline
\multirow{ 7}{*}{$n= 219 $} & $(1+2^{19})^{2066997275827785054568851212103878502402165134820}\cdot$\\
  &$(1+2^{63})^{15659070271422614049764024334120291684864887385}\cdot$ \\
& $(1+2^{107})^{8267989103311140218275404848415514009608660539280}\cdot$\\
 & $(1+2^{113})^{861248864928243772737021338376616042667568806175}\cdot$ \\
& $(1+2^{115})^{2239247048813433809116255479779201710935678896055}\cdot$\\
& $(1+2^{139})^{9802832608939360014892925802660670404346311615}\cdot$ \\
& $(1+2^{189})^{574165909952162515158014225584410695111712537450}$ \\
\hline 
\multirow{ 3}{*}{$n= 221 $}&$(1+2^{9})^{3989486683832532366484265709728099870565523986917322701595}\cdot$\\
&$ (1+2^{19})^{3989486683832532366484265709728099870565523986917322701595}\cdot$\\
&$(1+2^{29})^{3989486683832532366484265709728099870565523986917322701595}$\\
\hline

\multirow{ 5}{*}{$n= 223 $}&$(1+2^{17})^{1014848095919801109402176209373936431577327592183}\cdot$\\
& $(1+2^{91})^{1503478660621927569484705495368794713447892729160}\cdot$\\
&$(1+2^{103})^{1691413493199668515670293682289894052628879320305}\cdot$\\
&$(1+2^{113})^{563804497733222838556764560763298017542959773435}\cdot$\\
&$(1+2^{145})^{2631087656088373246598234616895390748533812276030}$\\
\hline
\multirow{ 7}{*}{$n= 225 $}&$(1+2^{29})^{64638048753257449056214060125}\cdot(1+2^{31})^{788456198653595814230254674000}\cdot$\\

& $(1+2^{49})^{5797472048923498634045990250}\cdot(1+2^{51})^{84063344709390730193666858625}\cdot$\\
&$(1+2^{85})^{4479864765077248944490083375}\cdot(1+2^{97})^{1809865365091208573573993683500}\cdot$\\

&$(1+2^{99})^{365703246128755015876741500}\cdot(1+2^{105})^{5302697068866947730212751750}\cdot$\\

& $(1+2^{115})^{68066717742077353015043250}\cdot (1+2^{131})^{5270429135384998758223627500}\cdot$\\

& $(1+2^{203})^{152315402012626464112662834750}\cdot(1+2^{207})^{828774981539291054730665424375}\cdot$\\
&$(1+2^{217})^{438130774024554946771130154075}\cdot(1+2^{219})^{94077160066622227834291750875}$\\
\hline
\multirow{ 2}{*}{$n= 227 $}&$(1+2^{5})^{59383142438657794704063431302361624145812280978139757521155099815}\cdot$\\
&$(1+2^{7})^{118766284877315589408126862604723248291624561956279515042310199630}$\\
\hline
\multirow{ 3}{*}{$n= 229$}&$(1+2^{9})^{147385697387219867509609565535361987370219507869771466051}\cdot$\\
&$(1+2^{21})^{1768628368646638410115314786424343848442634094437257592612}\cdot$\\
& $(1+2^{121})^{589542789548879470038438262141447949480878031479085864204}$\\
\hline 
\multirow{ 8}{*}{$n= 231$} & $(1+2^{41})^{459189704451443127647362594916155732882807466830218}\cdot$\\ 
&$ (1+2^{55})^{1567200356489566988557551518485173149770673948226}\cdot$\\ &$ (1+2^{69})^{2755138226708658765884175569496934397296844800981308}$ \\
& $(1+2^{91})^{1147974261128607819118406487290389332207018667075545}\cdot$\\ 
&$ (1+2^{131})^{535721321860016982255256360735515021696608711301921}\cdot$\\ 
&$ (1+2^{157})^{918379408902886255294725189832311465765614933660436}$ \\
& $(1+2^{159})^{1147974261128607819118406487290389332207018667075545}\cdot$\\
&$ (1+2^{209})^{2350800534734350482836327277727759724656010922339}$ \\

\hline
 
\multirow{ 4}{*}{$n= 233$} & $(1+2^{15})^{17038321523200602421013846948443997824061934891243478185488130}\cdot$\\
&$ (1+2^{23})^{2839720253866767070168974491407332970676989148540579697581355}$ \\
& $(1+2^{49})^{2839720253866767070168974491407332970676989148540579697581355}\cdot $\\
& $(1+2^{113})^{2839720253866767070168974491407332970676989148540579697581355}$ \\
\hline
\newpage
\hline

\multirow{ 7}{*}{$n= 235$} & $(1+2^{73})^{23191222403616672243206263595493280420250247335633250}\cdot$\\ 
&$ (1+2^{105})^{26891949382917205047973220552220931551141244250893875}$ \\
& $(1+2^{119})^{8141599354461172170487305304800832487960193213573375}\cdot$\\ 
&$ (1+2^{141})^{363659699007176721748930895080669461243705878250}$ \\
& $(1+2^{167})^{14062762521342024658114436435565074297385788277990375}\cdot$\\ 
&$ (1+2^{171})^{23684652667523409950508524523056967237702380257668000}$ \\
& $(1+2^{197})^{2556865912971277210566261170102740781342870595998250}$ \\
\hline

\multirow{ 7}{*}{$n= 237$} & $(1+2^{9})^{312620627852416761979153117722703908013332701360090433}\cdot$\\ 
&$ (1+2^{169})^{3647240658278195556423453040098212260155548182534388385}$ \\
& $(1+2^{173})^{1042068759508055873263843725742346360044442337866968110}\cdot$\\ 
& $ (1+2^{175})^{173678126584675978877307287623724393340740389644494685}$ \\
& $(1+2^{199})^{429087136268023006638053298835083795312417433239339810}\cdot$\\ 
& $ (1+2^{225})^{416827503803222349305537490296938544017776935146787244}$ \\
& $(1+2^{233})^{67961006054873209125902851678848675655072326382628355}$ \\

\hline
\multirow{ 6}{*}{$n= 239 $}&$(1+2^{59})^{137903930258717580167839711391793163262983138814261779866}\cdot$\\&$(1+2^{73})^{25172939650400828125875502873105101230544541212127150293}\cdot$\\&$ (1+2^{87})^{206855895388076370251759567087689744894474708221392669799}\cdot$\\&$ (1+2^{197})^{298791848893888090363652708015551853736463467430900523043}\cdot$\\&$ (1+2^{199})^{91935953505811720111893140927862108841988759209507853244}\cdot$\\&$ (1+2^{237})^{6566853821843694293706652923418722060142054229250560946}$\\

\hline 

\multirow{ 2}{*}{$n= 241 $}&$(1+2^{7})^{916414411031455987244373856003197176202834172185057465573737668369009}\cdot$\\&$ (1+2^{9})^{5498486466188735923466243136019183057217005033110344793442426010214054}$\\

\hline

\multirow{ 9}{*}{$n= 243 $}&$(1+2^{23})^{785013025560884300793728032156196770104389659275}\cdot$\\&
$(1+2^{31})^{22335098635205076340574688781013129860710668130}\cdot$\\&$ (1+2^{35})^{847419918806310249392392603750204044715198879050}\cdot$\\&$ (1+2^{61})^{132039847814006480719279777793636444176554243945}\cdot$\\&$ (1+2^{151})^{41135370735706828274307725135899531886322869175}\cdot$\\&$ (1+2^{155})^{725890705644164981068677385382926720473096714225}\cdot$\\&$ (1+2^{163})^{160944093106624814807082316216124023996297461525}\cdot$\\&$ (1+2^{177})^{387579652787382207086443128846992547582920417550}\cdot$\\&
$(1+2^{181})^{91968053203785608461189894980642299426455692300}$\\
\hline

\multirow{ 7}{*}{$n= 245 $}&$(1+2^{69})^{404534281273826986829987345146663806009193260698421162909645}\cdot$\\&$ (1+2^{117})^{652474647215849978758044105075264203240634291449066391789750}\cdot$\\&$ (1+2^{125})^{1096157407322627964313514096526443861444265609634431538206780}\cdot$\\&$ (1+2^{141})^{1057008928489676965588031450221928009249827552147487554699395}\cdot$\\&$ (1+2^{151})^{404534281273826986829987345146663806009193260698421162909645}\cdot$\\&$ (1+2^{165})^{1578988646262356948594466734282139371842334985306740668131195}\cdot$\\&$ (1+2^{167})^{404534281273826986829987345146663806009193260698421162909645}$\\

\hline
\multirow{ 6}{*}{$n= 247$} & $(1+2)^{134057357388441380704540286280333486890775035828909707815645} \cdot$ \\
& $(1+2^{9})^{130787665744820859223941742712520475015390278857472885673800}\cdot$ \\
& $(1+2^{35})^{16348458218102607402992717839065059376923784857184110709225}\cdot$ \\
& $(1+2^{71})^{85011982734133558495562132763138308760003681257357375687970}\cdot$ \\
& $(1+2^{147})^{81742291090513037014963589195325296884618924285920553546125}\cdot$ \\
& $(1+2^{195})^{15040581560654398810753300411939854626769882068609381852487}$ \\
\hline
 & $(1+2^{97})^{292527702190729434230102491312771097283901482612310325863937852070}\cdot$ \\
& $(1+2^{119})^{204769391533510603961071743918939768098731037828617228104756496449}\cdot$ \\
$n= 249$& $(1+2^{137})^{585055404381458868460204982625542194567802965224620651727875704140}\cdot$ \\
& $(1+2^{173})^{633810021413247107498555397844337377448453212326672372705198679485}\cdot$ \\
& $(1+2^{199})^{536300787349670629421854567406747011687152718122568930750552728795}$ \\
\hline

\end{longtable}
}

\section{Further comments}

One can go further than our choice of bound, namely $250$, by using our method. We have yet to encounter an exponent for which we cannot apply Algorithm~\ref{algo1}. If there is such an exponent~$n$ for which we cannot find $a_j$, $d_{i,j}$ as above, then we can involve cubics in the factorization of $-1\pmod{2^n-1}$. More precisely, we do something similar as above using $2^{a_i}+2^{b_i}+1$ and check if we can find such powers $a_i,b_i$ such that the number above is a quadratic nonresidue modulo $q_i$ and quadratic residue modulo $q_j$ for all $j\neq i$. The rest of Algorithm~\ref{algo1} runs unchanged.

While, via Carlitz' result, we know that any permutation can be decomposed as a composition of inverses and affine functions, it would also be interesting to  check whether one can modify our method in this paper to other exponents, other than the inverse and surely, the Gold $2^k+1$ exponents, to directly find their decomposition in quadratics, or quadratics and cubics, and we leave that to future work and the interested reader.

\vskip.5cm
\noindent
{\bf Acknowledgements.}
The first and the third-named authors  worked on this paper during visits to the Max Planck Institute for Software Systems in Saarbr\"ucken, Germany in Spring of 2022 and 2023. They thank Professor J. Ouaknine for the invitation and the Institute for hospitality and support. During the final stages of the preparation of this paper, the first-named author was a fellow at the Stellenbosch Institute for Advanced Study. He thanks this Institution for hospitality and support.

\end{document}